\documentclass[a4paper,12pt]{article}

\usepackage{amscd}
\usepackage{amsfonts}
\usepackage{amsmath}
\usepackage{amssymb}
\usepackage{amsthm}
\usepackage[T1]{fontenc} % changing encode
\usepackage{here} % [h] for tables
\usepackage{mathrsfs} % \mathscr
\usepackage{txfonts} % colneqq
\usepackage[all]{xy} % xypic
\usepackage{algorithm} % pseudocode
\usepackage{algpseudocode} %pseudocode
\usepackage{diagbox} % \diagbox

\allowdisplaybreaks

\theoremstyle{plain}
\newtheorem{thm}{Theorem}[section]

\newtheorem{prp}[thm]{Proposition}
\newtheorem{crl}[thm]{Corollary}
\theoremstyle{definition}

\newtheorem{exm}[thm]{Example}

\newcommand{\vs}[1][0.2]{\vspace{#1in}\noindent\ignorespaces}
\newcommand{\ba}{\begin{array*}}
\newcommand{\ea}{\end{array*}}
\newcommand{\be}{\begin{eqnarray*}}
\newcommand{\ee}{\end{eqnarray*}}
\newcommand{\bi}{\begin{itemize}}
\newcommand{\ei}{\end{itemize}}
\newcommand{\bb}{\vs\begin{itembox}}
\newcommand{\eb}{\end{itembox}}
\newcommand{\bc}{\begin{center}}
\newcommand{\ec}{\end{center}}
\newcommand{\bs}{\vs\begin{screen}}
\newcommand{\es}{\end{screen}}

\def\ens#1{{\mathchoice{\left\{ #1 \right\}}{\{ #1 \}}{\{ #1 \}}{\{ #1 \}}}}
\def\set#1#2{{\mathchoice{\left\{ #1 \middle| #2 \right\}}{\{ #1 \mid #2 \}}{\{ #1 \mid #2 \}}{\{ #1 \mid #2 \}}}}
\def\r#1{\text{\rm #1}}

\def\Bigv#1{\left| #1 \right|}
\def\v#1{{\mathchoice{\Bigv{#1}}{| #1 |}{| #1 |}{| #1 |}}}
\def\Bign#1{\left\| #1 \right\|}
\def\n#1{{\mathchoice{\Bign{#1}}{\| #1 \|}{\| #1 \|}{\| #1 \|}}}

\newcommand{\bA}{\mathbb{A}}

\newcommand{\bC}{\mathbb{C}}

\newcommand{\bN}{\mathbb{N}}

\newcommand{\bR}{\mathbb{R}}

\newcommand{\bZ}{\mathbb{Z}}

\newcommand{\rA}{\r{A}}

\newcommand{\rC}{\r{C}}

\newcommand{\C}{\bC}

\newcommand{\N}{\bN}

\newcommand{\R}{\bR}
\newcommand{\Z}{\bZ}

\newcommand{\Fp}{\mathbb{F}_p}

\newcommand{\Qp}{\mathbb{Q}_p}

\newcommand{\Zp}{\mathbb{Z}_p}

\newcommand{\Teichmuller}{Teichm\"uller }

\newcommand{\AD}{\r{AD}}
\newcommand{\IO}{\r{IO}}
\newcommand{\LS}{\r{LS}}
\newcommand{\opt}{\r{opt}}
\newcommand{\temp}{\r{temp}}

\algnewcommand\algorithmicbreak{{\bf break}}
\algnewcommand\Break{\algorithmicbreak{}}
\algnewcommand\algorithmiccontinue{{\bf continue}}
\algnewcommand\Continue{\algorithmiccontinue{}}

\title{$p$-adic Principal Component Analysis}
\author{Tomoki Mihara}
\date{}

\begin{document}

\maketitle
%\address
\begin{abstract}
We formulate a $p$-adic optimisation problem on matrix factorisation, and investigate a heuristic method for it analogous to PCA.
\end{abstract}

\tableofcontents
%\fn{11S80}{}

\section{Introduction}
\label{Introduction}

Principal component analysis (PCA) is a dimensionality reduction method based on linear algebra over $\R$, and is a useful tool to analyse continuous real variables. Although there are several techniques to directly analyse categorical variables, e.g.\ $\chi^2$ test, multiple correspondence analysis (MCA), and polychoric correlation, PCA is sometimes naively applied to them by embedding the set $C$ of levels into a Euclidean space. Then the resulting components might be a virtual quantity presenting formal real combinations of levels and not realised as categorical variables in $C$.

\vs
Especially when $C$ admits an algebraic structure, e.g.\ $C = \ens{0,1}$ equipped with boolean operations or $C = \Z/n \Z$ for an $n \in \N$ equipped with modular arithmetic, the resulting components lack the information of the original operation, unless the embedding reflects the structure in some way. When we are interested in the algebraic structure, then it is desired if we can embed $C$ into a space $\bA$ other than a Euclidean space in a way reflecting the algebraic structure to some extent and PCA works even for $\bA$.

\vs
Let $p$ be a prime number. In this paper, we formulate PCA for the case $\bA$ is $\Qp^D$ or $\Zp^D$ for a $D \in \N$, where $\Qp$ denotes the field of $p$-adic numbers and $\Zp$ denotes the ring of $p$-adic integers. Since $\Qp^D$ and $\Zp^D$ have algebraic structures quite different to that of a Euclidean space, $C$ might admit a good embedding to them or a direct product for multiple $p$'s reflecting the algebraic structure to some extent. For example, when $C = \ens{0,1}$ equipped with boolean operators, then \Teichmuller embedding $\Fp \hookrightarrow \Zp$ for $p = 2$ can be a good candidate, where $\Fp$ denote the finite field of integers modulo $p$.

\vs
The notion of $p$-adic numbers is originally introduced by K.\ Hensel in 1897 in \cite{Hen97}, and plays a central role in modern number theory. Recently, the $p$-adic numbers are used also in other branches of science, because of many significant similarities to and differences from the real numbers and the complex numbers. See Introduction of \cite{Mih26} for details. For the reader's convenience, we give a duplicated explanation here.

\vs
The $p$-adic numbers also appear in computer science. For example, S.\ Albeverio, A.\ Khrennikov, and B.\ Tirrozi studied $p$-adic neural network in \cite{AKT99} and \cite{KT00}, P.\ E.\ Bradley studied dendrograms and clusterings using $p$-adic numbers in \cite{Bra08} and \cite{Bra09}, and so on. Introduction of \cite{Bra25} explains the history well. Especially, $p$-adic equations and $p$-adic optimisation related to $p$-adic neural networks are recently studied as a frontier topic (cf.\ \cite{ZZ23}, \cite{ZZB24}, \cite{BMP25}, \cite{Zub25-1}, \cite{Zub25-2}, \cite{Ngu25}, \cite{Mih26}).

\vs
Although the $p$-adic numbers share many common properties with the real numbers, there are not so many studies on $p$-adic optimisation unlike real optimisation. At least, linear algebraic methods independent of the coefficient field and combinatorial methods work also for $p$-adic optimisation. On the other hand, the lack of $p$-adic counterparts of real methods based on gradients prevents from inventing $p$-adic counterparts of real optimisation methods. Although Newton's method also works for multivariable polynomials in $\Qp$, it is not applicable to a $p$-adic optimisation problem for functions without zeros, e.g.\ $p$-adic linear regression or $p$-adic polynomial regression.

\vs
Another problem is that when we deal with a loss function $\epsilon$ of the form $X \to [0,\infty)$ for a subset $X$ of $\Qp$, the differential of $\epsilon$ does not naturally make sense. In order to differentiate a function $f$, we need to assume that the domain of $f$ and the codomain of $f$ admit ambient spaces with a compatible arithmetic structure. In the setting of $\epsilon$, the ambient space of $X$ is $\Qp$, while that of $[0,\infty)$ is $\R$. Since we do not have a definition of arithmetic of a pair of a $p$-adic number and a real number, the usual formulation fails.

\vs
We note that it is not usually useful to formally regard $p$-adic absolute values as $p$-adic numbers rather than real numbers in order to formulate the loss function as a $p$-adic function, because then the $p$-adic absolute value function is discontinuous at $0$ and is not bounded on $\Zp$.

\vs
In addition, even if we extend the notion of a differential so that it vanishes at any point where $\epsilon$ is locally constant, it frequently occurs that $\epsilon$ is locally constant at almost all points of the domain of $\epsilon$. Therefore, it is very difficult to make use of a differential to determine a small step in optimisation.

\vs
Another problem is that a symmetric matrix is not necessarily diagonalisable in the $p$-adic setting. Giving a good criterion of a $p$-adic matrix except for one applicable to a general field is itself a difficult problem. For example, see \cite{Mih16} Theorem 3.23 for the repetitive reduction method of a $p$-adic matrix based on the reduction $\Zp \twoheadrightarrow \Fp$ and Hensel lifting of idempotents, and \cite{Mih17} Theorem 3.16 for its extension to a unitary Banach $p$-adic representation of a topological monoid.

\vs
By the lack of the diagonalisability of a symmetric matrix, the standard PCA technique based on diagonalisation of the covariance matrix does not work in the $p$-adic setting. Moreover, we do not have a natural $p$-adic distribution which can play a role of significant non-uniform distributions such as normal distribution in the $p$-adic setting, and we mainly consider a uniform distribution on $\Zp$. Unlike a normal distribution in the real setting, the statistical average or mode of a uniform distribution on $\Zp$ is not so meaningful, and hence normalisation or standardisation, which is useful when we connect correlation to the inner product, does not work in this setting.

\vs
Furthermore, the $p$-adic standard inner product does not necessarily satisfy the non-degeneracy $\langle \vec{v}, \vec{v} \rangle = 0 \Leftrightarrow \vec{v} = \vec{0}$. Although there is a classical study of the $p$-adic Hilbert space structure given by the standard inner product by G.\ K.\ Kalisch (cf.\ \cite{Kal47}), it is applicable only to a subspace satisfying a very strong condition called {\it normality} only in the case where $p \neq 2$ and the base field is closed under square root like $\C$. Of course, $\Qp$ does not satisfy the condition of the base field.

\vs
As we have shown many problems on the $p$-adic inner product, it is very difficult to connect correlation to the inner product, and even the covariant matrix does not perform well in the $p$-adic setting. Instead, we introduce a $p$-adic counterpart of orthogonality based on the relation between perpendicular and the nearest point. Although the notion of $p$-adic orthogonality has ever formulated in terms of the $\ell^{\infty}$-norm direct sum (cf.\ \cite{BGR84} Definition 2.4.1/1) and an inner product (cf.\ \cite{Kal47} \S 3), we formulate it in another way because both are not applicable to our setting. Using the $p$-adic orthogonality, we introduce a $1$-dimensional projection and an orthogonalisation in the $p$-adic setting, and formulate a $p$-adic counterpart of PCA by iteration of $1$-dimensional projections and orthogonalisations.

\vs
The $p$-adic PCA gives a low rank approximation of a matrix given by matrix factorisation. Here, we consider the $\ell^q$-norm for a $q \in [1,\infty)$ when we formulate the loss function for the approximation. If we considered the $\ell^{\infty}$-norm, then dimensionality reduction by elementary divisor theory, i.e.\ matrix factorisation based on Smith normal form would work. Smith normal form is a normal form of a matrix over a PID, which is named after H.\ J.\ S.\ Smith for the study of linear equations (cf.\ \cite{Smi61}), and extended to wider classes of rings (cf.\ \cite{Sta16} Theorem 2.1 for references for various related results). Since computation of Smith normal form is executed in an elementary way similar to Gauss elimination, it would be better if we could apply it for our setting.

\vs
The reason why we consider the $\ell^q$-norm rather than the $\ell^{\infty}$-norm is because the use of $\ell^{\infty}$-norm prevents from application of the $p$-adic PCA to anomaly detection tasks in unsupervised settings. If we know that the $\ell^{\infty}$-norms of all anomaly data are always smaller than those of the majority of normal data, then the $\ell^{\infty}$-norm works well here. However, $\ell^{\infty}$-norms frequently coincide in the $p$-adic setting because they lie in a sparse subset of $\R$, and hence it is not practical to assume such a domination property. As long as we consider the $\ell^q$-norm, we can apply the $p$-adic PCA to anomaly detection tasks, as experimental results show.

\vs
There are plenty of preceding studies on binary matrix factorisation (BMF), which gives low rank approximation of a boolean matrix by matrix factorisation based on boolean operators such as the logical conjunction $\lor$ and the logical xor $\oplus$ (cf.\ \cite{BV10}, \cite{AV21}, \cite{MDLRB21}, \cite{FM23}). BMF gives dimensionality reduction of categorical data respecting the algebraic structure for the case where $C = \ens{0,1}$ equipped with boolean operations. 

\vs
Our study is not restricted to the case $p = 2$, and boolean data can also be regarded as a $p$-adic data through the embedding $\ens{0,1} \hookrightarrow \Qp$. Although the embedding does not reflect many of boolean operations unless $p = 2$, it is possible that the total disconnectedness of $\Qp$ matches given categorical data more than the connectedness of $\R$ because of the discreteness of $C$.

\vs
We briefly explain contents of this paper. In \S \ref{Convention and Preliminaries}, we introduce convention of this paper and recall elementary properties of normed vector spaces. In \S \ref{p-adic Orthogonality}, we introduce a $p$-adic counterpart of orthogonality (cf.\ \S \ref{1-Dimensional Projection}) and formulate an orthogonalisation (cf.\ \S \ref{Orthogonal System}). In \S \ref{p-adic Low Rank Approximation}, we formulated two $p$-adic PCAs (cf.\ \S \ref{Non-reduced p-adic PCA} and \S \ref{Reduced p-adic PCA}) and $p$-adic conterparts of line search and coordinate descent (cf.\ \S \ref{p-adic Line Search and p-adic Coordinate Descent}) useful to check whether the $p$-adic PCAs produce approximately optimal solutions to some extent or not. In \S \ref{Experiment}, we exhibit experimental results.

\section{Convention and Preliminaries}
\label{Convention and Preliminaries}

We denote by $\N$ the set of non-negative integers. For a $d \in \Z$, we set $\N_{< d} \coloneqq \N \cap [0,d)$, and $\N_{\leq d} \coloneqq \N \cap [0,d]$. For sets $X$ and $Y$, we denote by $X^Y$ the set of maps $Y \to X$. We note that every $d \in \N$ is identified with $\N_{< d}$ in set theory, and hence $X^d$ formally means $X^{\N_{< d}}$, which is naturally identified with the set of $d$-tuples in $X$.

\vs
Throughout this paper, let $D$ and $E$ denote positive integers. We use $D$ for the dimension of vectors which we handle as input data, and $E$ for accuracy hyperparameter.

\vs
When we write a pseudocode, a for-loop along a subset of $\N$ denotes the loop of the ascending order, and a for-loop along a general set $S$ denotes a loop in an arbitrary order.

\vs
A {\it valuation field} is a field $k$ equipped with a map $\v{\cdot} \colon k \to [0, \infty)$ called a {\it (multiplicative) valuation} satisfying the following:
\bi
\item[(1)] For any $(c_0,c_1) \in k^2$, the inequality $\v{c_0-c_1} \leq \max \ens{\v{c_0}, \v{c_1}}$ holds.
\item[(2)] For any $(c_0,c_1) \in k^2$, the equality $\v{c_0 c_1} = \v{c_0} \ \v{c_1}$ holds.
\item[(3)] For any $c \in k$, the equality $\v{c} = 0$ holds if and only if $c = 0$.
\item[(4)] The equality $\v{1} = 1$ holds.
\ei
Let $k$ be a valuation field. Its closed unit ball $O_k \coloneqq \set{c \in k}{\v{c} \leq 1}$ centred at $0$ forms a subring of $k$, which admits a unique maximal ideal $m_k \coloneqq \set{c \in O_k}{\v{c} < 1}$. We call $O_k/m_k$ {\it the residue field of $k$}.

\vs
A {\it normed $k$-vector space} is a $k$-vector space $V$ equipped with a map $\n{\cdot} \colon V \to [0,\infty)$ called a {\it norm} satisfying the following:
\bi
\item[(1)] For any $(\vec{v}_0,\vec{v}_1) \in V^2$, the inequality $\n{\vec{v}_0 - \vec{v}_1} \leq \n{\vec{v}_0} + \n{\vec{v}_1}$ holds.
\item[(2)] For any $(c,\vec{v}) \in k \times V$, the equality $\n{c \vec{v}} = \v{c} \ \n{\vec{v}}$ holds.
\item[(3)] For any $\vec{v} \in V$, the equality $\n{\vec{v}} = 0$ holds if and only if $\vec{v} = 0$.
\ei
We note that the condition (1) is usually replaced by the following stronger condition called {\it the strong triangular inequality} in the $p$-adic setting:
\bi
\item[(1)'] For any $(\vec{v}_0,\vec{v}_1) \in V^2$, the inequality $\n{\vec{v}_0 - \vec{v}_1} \leq \max \ens{\n{\vec{v}_0},\n{\vec{v}_1}}$ holds.
\ei
However, we do not assume the condition (1)' because we are interested in the $\ell^q$-norm on $k^D$ for a $q \in [1,\infty)$, which does not satisfy the condition (1) unless $D \leq 1$. Nevertheless, the condition (1) implies the condition (1)' for the $1$-dimensional setting.

\begin{prp}
\label{non-Archimedean property}
Let $V$ be a normed $k$-vector space of dimension $\leq 1$. Then $V$ satisfies the condition (1)'.
\end{prp}

\begin{proof}
Let $(\vec{v}_0,\vec{v}_1) \in V^2$. Since $V$ is of dimension $\leq 1$, there exists a $(\vec{v},c_0,c_1) \in V \times k^2$ such that $\vec{v}_0 = c_0 \vec{v}$ and $\vec{v}_1 = c_1 \vec{v}$. We have
\be
\n{\vec{v}_0 - \vec{v}_1} & = & \n{(c_0 - c_1) \vec{v}} = \v{c_0 - c_1} \ \n{\vec{v}} \leq \max \ens{\v{c_0},\v{c_1}} \n{\vec{v}} \\
& = & \max \ens{\v{c_0} \ \n{\vec{v}},\v{c_1} \ \n{\vec{v}}} = \max \ens{\n{c_0 \vec{v}},\n{c_1 \vec{v}}} = \max \ens{\n{\vec{v}_0},\n{\vec{v}_1}}.
\ee
\end{proof}

For example, $k$ itself forms a normed $k$-vector space. We always equip a normed $k$-vector space $V$ with the metric
\be
V^2 & \to & [0,\infty) \\
(\vec{v}_0,\vec{v}_1) & \mapsto & \n{\vec{v}_0 - \vec{v}_1}.
\ee
We say that $k$ is a {\it local field} if the underlying metric of $k$ is complete, $\v{k^{\times}} \coloneqq \set{\v{c}}{c \in k^{\times}}$ is a free $\Z$-submodule of the multiplicative group $(0,\infty)$ of rank $1$, and the residue field of $k$ is a finite field.

\vs
Throughout the paper, we fix a prime number $p$. We denote by $\Fp$ the finite field of integers modulo $p$, $\Qp$ the field of $p$-adic numbers, which forms a local field with respect to any $p$-adic absolute value $\v{\cdot} \colon \Qp \to [0,\infty)$, and $\Zp$ the ring $O_{\Qp}$ of $p$-adic integers. We naturally identify $\Fp$ with the residue field of $\Qp$.

\vs
We note that the $p$-adic absolute value with the normalisation $\v{p} = p^{-1}$ is frequently denoted by $\v{\cdot}_p$, but we do not use the convention. Instead, for an $\epsilon \in (0,1)$, we temporarily denote by $\v{\cdot}_{\epsilon}$ the $p$-adic absolute value $\Qp \to [0,\infty)$ with the normalisation $\v{p}_{\epsilon} = \epsilon$. For a $(q,\epsilon) \in (0,\infty) \times (0,1)$, we temporarily denote by $\n{\cdot}_{\epsilon,q}$ the $\ell^q$-norm on $\Qp^D$ with respect to $\v{\cdot}_{\epsilon}$. Then we have
\be
\n{\cdot}_{\epsilon,q} = \n{\cdot}_{\epsilon^q,1}^{1/q}
\ee
for any $(q,\epsilon) \in (0,\infty) \times (0,1)$. Therefore, these norms are essentially reduced to the $\ell^1$-norm with an arbitrary normalisation. In particular, the choice of $\epsilon$ and $q$ is not so notable in our context of optimisation, and hence we will not use the conventions $\v{\cdot}_{\epsilon}$ or $\n{\cdot}_{\epsilon,q}$.

\section{$p$-adic Orthogonality}
\label{p-adic Orthogonality}

We introduce $p$-adic counterpart of the orthogonality based on the relation between perpendicular and the nearest point. The reader should be careful that this notion is not equivalent to the $p$-adic orthogonality based on the $\ell^{\infty}$-norm direct sum (cf.\ \cite{BGR84} Definition 2.4.1/1) or the $p$-adic orthogonality based on an inner product (cf.\ \cite{Kal47} \S 3).

\subsection{$1$-Dimensional Projection}
\label{1-Dimensional Projection}

Let $(X,d)$ be a metric space, and $S$ a non-empty subset of $X$. An $s \in S$ is said to be {\it the nearest neighbour of an $x \in X$ in $S$} if the equality
\be
d(x,s) = d(x,S) \coloneqq \inf_{s' \in S} d(x,s')
\ee
holds. Although the nearest neighbour of an $x \in X$ is not necessarily unique in $S$, we have the following existence criterion:

\begin{prp}
\label{nearest neighbour}
Suppose that $S$ satisfies Bolzano--Weierstrass property, i.e.\ every bounded sequence in $S$ has a subsequence converging in $S$. Then for any $x \in X$, there exists the nearest neighbour of $x$ in $S$.
\end{prp}

\begin{proof}
Since $\R$ is first countable, there exists an $(s_i)_{i \in \N} \in S^{\N}$ such that $(d(x,s_i))_{i \in \N}$ converges to $d(x,S)$. The convergence of $(d(x,s_i))_{i \in \N}$ implies the boundedness of $(d(x,s_i))_{i \in \N}$, and hence the boundedness of $(s_i)_{i \in \N}$. Since $S$ satisfies Bolzano--Weierstrass property, $(s_i)_{i \in \N}$ admits a subsequence $(s'_i)_{i \in \N} \in S^{\N}$ converging to an $s \in S$. Since $(s'_i)_{i \in \N}$ is a subsequence of $(s_i)_{i \in \N}$, $(d(x,s'_i))_{i \in \N}$ converges to $d(x,S)$. Since $(s'_i)_{i \in \N}$ converges to $s$ and $d$ is continuous, we have $d(x,s) = d(x,S)$. Thus, $s$ is the nearest neighbour of $x$ in $S$.
\end{proof}

Let $k$ be a local field. We give a sufficient condition of Bolzano--Weierstrass property, which plays a key role in this subsection.

\begin{prp}
\label{Bolzano--Weierstrass property}
Every normed $k$-vector space of dimension $\leq 1$ satisfies Bolzano--Weierstrass property.
\end{prp}

\begin{proof}
Since every closed ball in $k$ is compact, $k^1 \cong k$ satisfies Bolzano--Weierstrass property. Since $k^0$ is a singleton, it also satisfies Bolzano--Weierstrass property.

\vs
Let $W$ be a normed $k$-vector space of dimension $\leq 1$. By Proposition \ref{non-Archimedean property}, $W$ satisfies the strong triangular inequality. By the finiteness of $\dim_k W$, $W$ is $k$-linearly isomorphic to $k^{\dim_k W}$. By the argument above, $k^{\dim_k W}$ satisfies Bolzano--Weierstrass property. By the strong triangular inequality, any $k$-linear isomorphism $W \to k^{\dim_k W}$ is an admissible isomorphism by \cite{BGR84} Corollary 2.1.8/3 and \cite{BGR84} Proposition 2.3.3/4. Therefore, $W$ satisfies Bolzano--Weierstrass property.
\end{proof}

Let $V$ be a normed $k$-vector space. By Proposition \ref{nearest neighbour} and Proposition \ref{Bolzano--Weierstrass property}, we obtain the following:

\begin{crl}
\label{nearest neighbour 2}
For any $(\vec{v}_0,\vec{v}_1) \in V^2$, there exists the nearest neighbour of $\vec{v}_0$ in $k \vec{v}_1$.
\end{crl}

Let $(\vec{v}_0,\vec{v}_1) \in V^2$. A $\vec{w} \in V$ is said to be a {\it $\vec{v}_1$-component of $\vec{v}_0$} if $\vec{w}$ is the nearest neighbour of $\vec{v}_0$ in $k \vec{v}_1$, which exists by Corollary \ref{nearest neighbour 2}, and is said to be a {\it $\vec{v}_1$-orthogonal component of $\vec{v}_0$} or an {\it orthogonalisation of $\vec{v}_0$ by $\vec{v}_1$} if $\vec{v}_0 - \vec{w}$ is a $\vec{v}_1$-component of $\vec{v}_0$. We say that $\vec{v}_0$ is {\it orthogonal to $\vec{v}_1$} if $\vec{v}_0$ itself is a $\vec{v}_1$-orthogonal component of $\vec{v}_0$, i.e.\ $0 \in V$ is a $\vec{v}_1$-component of $\vec{v}_0$.

\begin{exm}
Suppose $V = k^2$ equipped with the $\ell^1$-norm. Set $\vec{v}_0 \coloneqq (1,0)$, $\vec{v}_1 \coloneqq (0,1)$, and $\vec{v}_2 \coloneqq (1,1)$.
Although $\vec{v}_0$ and $\vec{v}_1$ are orthogonal to $\vec{v}_2$, $\vec{v}_2$ is not orthogonal to $\vec{v}_0$ or $\vec{v}_1$ by $\n{\vec{v}_2 - \vec{v}_0} = \n{\vec{v}_2 - \vec{v}_1} = 1 < 2 = \n{\vec{v}_2}$, and $\vec{v}_0 + \vec{v}_1 = \vec{v}_2$ is not orthogonal to $\vec{v}_2$ by $\n{\vec{v}_2 - \vec{v}_2} = 0 < 2 = \vec{v}_2$. Therefore, orthogonality is not symmetric, and the set $\vec{v}_2^{\perp}$ of elements in $V$ to which $\vec{v}_2$ is orthogonal does not form a $k$-linear subspace of $V$. We note that orthogonality is stable under scalar multiplication.
\end{exm}

Let $\pi$ be a uniformiser of $k$, i.e.\ an element of $k^{\times}$ whose absolute value is the maximum of $\v{k^{\times}} \cap (0,1)$. Suppose that $V$ is $k^D$ equipped with the $\ell^q$-norm for a $q \in [1,\infty)$, and all entries of $\vec{v}_0$ and $\vec{v}_1$ admit finite $\pi$-adic expansions. Then, the loss function
\be
k & \to & [0,\infty) \\
c & \mapsto & \n{\vec{v}_0 - c \vec{v}_1}^q
\ee
is rapidly computed from the multiset of the ratio of entries of $\vec{v}_0$ and $\vec{v}_1$ at the same positions, and hence a $\vec{v}_1$-component of $\vec{v}_0$ is computed by trie tree algorithm. The optimisation problem
\be
\begin{array}{ll}
\r{minimise} & \n{\vec{v}_0 - c \vec{v}_1} \\
\r{subject to} & c \in k
\end{array}
\ee
has a solution in $k$ by Corollary \ref{nearest neighbour 2}, which can be computed approximately modulo $\pi^E$ by depth first search on the trie tree.

\vs
For an $e \in \N_{\leq E}$, we set $\epsilon_e \coloneqq p^{-eq} \in (0,1]$ if $e < E$ and $\epsilon_e \coloneqq 0 \in [0,1]$ if $e = E$. Here are pseudocodes for the solution-finding algorithm for the case where $k = \Qp$ and $\vec{v}_0$ and $\vec{v}_1$ are rescaled by scalar multiplication and approximated up to modulo $p^E$ so that they lie in $\N_{p^E}^D \subset \Zp^D$:

\begin{figure}[H]
\begin{algorithm}[H]
\caption{Pre-computation of the ratio data $\vec{\rho}$ and the valuation data $\vec{\nu}$ for $\vec{v}_0 = (v_{0,d})_{d=0}^{D-1} \in \N_{p^E}^D$ and $\vec{v}_1 = (v_{1,d})_{d=0}^{D-1} \in \N_{p^E}^D$}
\label{RatioValuation}
\begin{algorithmic}[1]
\Function {RatioValuation}{$p,D,E,\vec{v}_0,\vec{v}_1$}
	\State $\vec{\rho} \gets$ the empty array
	\State $\vec{\nu} \gets$ the empty array
	\ForAll {$d \in \N_{< D}$}
		\State $\nu_0 \gets$ the additive $p$-adic valuation of $v_{0,d}$
		\State $\nu_1 \gets$ the additive $p$-adic valuation of $v_{1,d}$
		\If {$v_{1,d} \neq 0$ and $\nu_0 \geq \nu_1$}
			\State Append $\frac{v_{0,d}}{v_{1,d}} \bmod p^{E - \nu_1}$ to $\vec{\rho}$ and $\nu_1$ to $\vec{\nu}$
		\Else
			\State Append $0$ to $\vec{\rho}$ and $E$ to $\vec{\nu}$ \Comment{dummy entries}
		\EndIf
	\EndFor
	\State \Return $(\vec{\rho},\vec{\nu})$
\EndFunction
\end{algorithmic}
\end{algorithm}
\end{figure}

Here, the residue $\frac{v_{0,d}}{v_{1,d}} \bmod p^{E - \nu_1}$ in the line 8 can be computed by the multiplication of $p^{\nu_0 - \nu_1}$, $\frac{v_{0,d}}{p^{\nu_0}}$, and the modular inverse of $\frac{v_{1,d}}{p^{\nu_1}}$ modulo $p^{E - \nu_1}$.

\begin{figure}[H]
\begin{algorithm}[H]
\caption{Construction of a trie tree for ratio data $\vec{\rho} = (\rho_d)_{d=0}^{D-1} \in \N_{p^E}^D$ and valuation data $\vec{\nu} = (\nu_d)_{d=0}^{D-1} \in \N_{\leq E}^D$}
\label{trie tree}
\begin{algorithmic}[1]
\Function {TrieTree}{$p,D,E,\vec{\rho},\vec{\nu}$}
	\State $T \gets$ an $\N_{< p}$-labeled $\R$-weighted rooted tree such that $V_T = \ens{\ast_T}$ and $w_T(\ast_T) = 0$ \Comment{variable for a trie tree}
	\ForAll {$d \in \N_{< D}$}
		\State $n \gets \ast_T$
		\State $w_T(n) \gets w_T(n) + \epsilon_{\nu_d}$
		\ForAll {$e \in \N_{< E - \nu_d}$}
			\State $r \gets \rho_d \bmod p$
			\State $\rho_d \gets \frac{\rho_d - r}{p}$
			\If {There is no child node of $n$ of label $r$}
				\State Append to $T$ a new directed edge from $n$ to a new node of weight $0$ and label $r$
			\EndIf
			\State $n \gets$ the unique child node of $n$ of label $r$
			\State $w_T(n) \gets w_T(n) - \epsilon_{\nu_d + e} + \epsilon_{\nu_d + e + 1}$
		\EndFor
	\EndFor
	\State \Return $T$
\EndFunction
\end{algorithmic}
\end{algorithm}
\end{figure}

Here, for an $\R$-weighted rooted tree $T$, we denote by $\ast_T$ the root of $T$, by $V_T$ the set of vertices of $T$, and by $w_T$ the weight map $V_T \to \R$.

\begin{figure}[H]
\begin{algorithm}[H]
\caption{Solving the optimisation problem for $\vec{v}_0 = (v_{0,d})_{d=0}^{D-1} \in \N_{p^E}^D$ and $\vec{v}_1 = (v_{1,d})_{d=0}^{D-1} \in \N_{p^E}^D$}
\label{trie tree dfs}
\begin{algorithmic}[1]
\Function {TrieTreeDFS}{$p,D,E,\vec{v}_0,\vec{v}_1$}
	\State $(\vec{\rho},\vec{\nu}) \gets$ \Call{RatioValuation}{$p,D,E,\vec{v}_0,\vec{v}_1$}
	\State $T \gets$ \Call{TrieTree}{$p,D,E,\vec{\rho},\vec{\nu}$}
	\State $\temp \gets ((0,w_T(\ast)))$ \Comment{variable for a stack of temporary values}
	\State $\opt = (c_{\opt},w_{\opt}) \gets (0,\infty)$ \Comment{variable for an optimal value}
	\Function {DFS}{$n$}
		\If {There is a child node of $n$}
			\ForAll {child node $n'$ of $n$}
				\State $r \gets$ the label of $n'$
				\State $e \gets$ the length of $\temp$ minus $1$
				\State $(c,w) \gets$ the last entry of $\temp$
				\State Append $(c + r p^e, w + w_T(n'))$ to $\temp$
				\State \Call {DFS}{$n'$}
			\EndFor
		\Else
			\State $(c,w) \gets$ the last entry of $\temp$
			\If {$w_{\opt} > w$}
				\State $\opt \gets (c,w)$
			\EndIf
		\EndIf
		\State Pop the last entry of $\temp$
	\EndFunction
	\State \Call{DFS}{$\ast_T$}
	\State \Return $c_{\opt}$
\EndFunction
\end{algorithmic}
\end{algorithm}
\end{figure}

Practically speaking, when $q$ is a positive integer, it is good to rescale the weight map by $p^{(E-1)q}$ so that the whole process can be handled by integer arithmetic without numerical error.

\subsection{Orthogonal System}
\label{Orthogonal System}

Let $J$ be a finite set, and $X$ a sequence $(\vec{x}_j)_{j \in J} \in V^J$ in $V$ indexed by $J$. A $\vec{v} \in V$ is said to be {\it orthogonal to $X$} if $\vec{v}$ is orthogonal to $\vec{x}_j$ for any $j \in J$. We say that $X$ is an {\it orthogonal system} if for any $(j_0,j_1) \in J^2$ with $j_0 \neq j_1$, $\vec{x}_{j_0}$ is orthogonal to $\vec{x}_{j_1}$. We consider the iteration of an orthogonalisation by $\vec{x}_j$ for each $j \in J$ especially when $X$ is an orthogonal system, and call it an {\it orthogonalisation by $X$}. Here is a pseudocode for an orthogonalisation:

\begin{figure}[H]
\begin{algorithm}[H]
\caption{Orthogonalisation of $\vec{v} \in \N_{< p^E}^D$ by $X = (\vec{x}_j)_{j \in J} \in (\N_{p^E}^D)^J$}
\label{orthogonalisation}
\begin{algorithmic}[1]
\Function {Orthogonalisation}{$p,D,E,X,\vec{v}$}
	\ForAll {$j \in J$} \Comment{process depending on enumeration of $J$}
		\State $c \gets$ \Call{TrieTreeDFS}{$p,D,E,\vec{v},\vec{x}_j$}
		\If {$c = 0$}
			\State \Continue
		\EndIf
		\State $\vec{v} \gets (\vec{v} - c \vec{x}_j) \bmod p^E$
	\EndFor
	\State \Return $\vec{v}$
\EndFunction
\end{algorithmic}
\end{algorithm}
\end{figure}

\begin{exm}
\label{non-stability of orthogonality}
Suppose that the valuation of $k$ is normalised so that $2 \v{\pi} > 1$, e.g.\ $k = \Qp$ with the normalisation $\v{p} = 0.6$, and $V$ is $k^5$ equipped with the $\ell^1$-norm. Set $\vec{v} \coloneqq (0,\pi,\pi,- \pi,0)$, $\vec{x}_0 \coloneqq (1,\pi,\pi,0,0)$, and $\vec{x}_1 \coloneqq (1,0,0,\pi,\pi)$. Then by $2^{-1} < \v{\pi} < 1$, $(\vec{x}_0,\vec{x}_1) \in V^2$ is an orthogonal system, $\vec{v}$ is orthogonal to $\vec{x}_1$, $\vec{v} - \vec{x}_0 = (-1,0,0,- \pi,0)$ is an $\vec{x}_0$-orthogonalisation of $\vec{v}$, and $\vec{v} - \vec{x}_0$ is not orthogonal to $\vec{x}_1$ by $n{(\vec{v} - \vec{x}_0) + \vec{x}_1} = \v{\pi} < 1 + \v{\pi} = \n{\vec{v} - \vec{x}_0}$. Therefore, an orthogonalisation by $\vec{x}_0$ does not preserve the orthogonality to $\vec{x}_1$ despite of the orthogonality of $(\vec{x}_0,\vec{x}_1)$. This is the reason why we noted the dependency on enumeration of $J$ in the line 2 of Algorithm \ref{orthogonalisation}.
\end{exm}

By Example \ref{non-stability of orthogonality}, a single application of an orthogonalisation by an orthogonal system $X$ does not necessarily yields a vector orthogonal to $X$, and hence we instead need repetition of orthogonalisations by $X$ until the norm becomes stable or the number of repetition reaches a fixed threshold, in order to obtain a vector approximately orthogonal to $X$, unlike Gram--Schmidt orthonormalisation and its $p$-adic counterpart in \cite{Kal47} Theorem 5.

\vs
In particular, when we want to translate a system $X = (\vec{x}_j) \in V^J$ into an approximately orthogonal system, we need repetition of the process given by an iteration of an orthogonalisation of all entries other than $\vec{x}_j$ by $\vec{x}_j$ for each $j \in J$ until $X$ becomes an orthogonal system or the number of repetition reaches a fixed threshold $T_{\IO}$, and call the repetition an {\it iterated orthogonalisation of $X$}. Here is a pseudocode for an iterated orthogonalisation:

\begin{figure}[H]
\begin{algorithm}[H]
\caption{Iterated orthogonalisation of $X = (\vec{x}_j)_{j \in J} \in (\N_{p^E}^D)^J$}
\label{iterated orthogonalisation}
\begin{algorithmic}[1]
\Function {IteratedOrthogonalisation}{$p,D,E,X$}
	\State $\temp \gets -1$
	\State $t \gets 0$
	\While {$t < T_{\IO}$}
		\ForAll {$j \in J$} \Comment{process depending on enumeration of $J$}
			\If{$t = T_{\IO}$}
				\State \Break
			\EndIf
			\State $t \gets t + 1$
			\State $b \gets$ False \Comment{variable monitoring update of $X$}
			\ForAll {$j' \in J$}
				\If{j = j'}
					\State \Continue
				\EndIf
				\State $c \gets$ \Call{TrieTreeDFS}{$p,D,E,\vec{x}_{j'},\vec{x}_j$}
				\If {$c = 0$}
					\State \Continue
				\EndIf
				\State $b \gets$ True
				\State $\vec{x}_{j'} \gets (\vec{x}_{j'} - c \vec{x}_j) \bmod p^E$
			\EndFor
			\If {$b$}
				\State $\temp \gets t$
			\ElsIf {$t = \temp + \# J$}
				\State \Break \Comment{the case where $X$ is an orthogonal system}
			\EndIf
		\EndFor
	\EndWhile
	\State \Return $X$
\EndFunction
\end{algorithmic}
\end{algorithm}
\end{figure}

When $q$ is a positive integer, Algorithm \ref{trie tree dfs} can be handled by integer arithmetic without numerical error. In particular, the norm satisfies the well-foundedness, and hence Algorithm \ref{iterated orthogonalisation} terminates even when we formally set $T_{\IO} = \infty$.

\section{$p$-adic Low Rank Approximation}
\label{p-adic Low Rank Approximation}

We continue to consider the case where $V$ is $k^D$ equipped with the $\ell^q$-norm for a $q \in [1,\infty)$. Let $I$ be a finite set, $Y = (\vec{y}_i)_{i \in I} \in V^I$, and $D_{-} \in \N_{< D}$. We consider the following optimisation problem:
\be
\begin{array}{ll}
\r{minimise} & \sum_{i \in I} \n{\vec{y}_i - \sum_{d=0}^{D_{-}-1} c_{d,i} \vec{x}_d} \\
\r{subject to} & C = ((c_{d,i})_{i \in I})_{d=0}^{D_{-}-1} \in (k^I)^{D_{-}} \land X = (\vec{x}_d)_{d=0}^{D_{-}-1} \in V^{D_{-}}
\end{array}
\ee
If we were considering the $\ell^{\infty}$-norm, then the optimisation problem would be immediately solved by matrix factorisation based on Smith normal form. However, the optimisation problem for the $\ell^q$-norm requires much more complicated consideration.

\vs
Rescaling $Y$ by scalar multiplication, we may assume $Y \in (O_k^D)^I$ and the optimisation problem is reduced to the following optimisation problem:
\be
\begin{array}{ll}
\r{minimise} & \sum_{i \in I} \n{\vec{y}_i - \sum_{d=0}^{D_{-}-1} c_{d,i} \vec{x}_d} \\
\r{subject to} & C = ((c_{d,i})_{i \in I})_{d=0}^{D_{-}-1} \in (O_k^I)^{D_{-}} \land X = (\vec{x}_d)_{d=0}^{D_{-}-1} \in (O_k^D)^{D_{-}}
\end{array}
\ee
An approximate solution $X$ plays a role for a new coordinate system for the space where the sample data $\set{\vec{y}_i}{i \in I}$ normally lie in. When $D_{-}$ is significantly small and the corresponding loss is sufficiently small for a fixed $X$, then computation of the coefficient matrix $C$ gives dimensionality reduction of the sample data.

\vs
In order to obtain a heuristic solution $(C,X)$ for $(D_{-},Y)$, we consider the following recursive approach:
\bi
\item[(1)] If $D_{-} = 0$, then return the pair $(C,X) \coloneqq ((),())$ of the trivial elements $() \in (O_k^I)^{D_{-}} = (O_k^I)^0$ and $() \in (O_k^D)^{D_{-}} = (O_k^D)^0$.
\item[(2)] Choose an $\vec{x} \in O_k^D$ in some way, and compute a $\vec{c} = (c_i)_{i \in I} \in O_k^I$ such that $\vec{y}_i - c_i \vec{x}$ is orthogonal to $\vec{x}$, i.e.\ $c_i \vec{x}$ is the $\vec{x}$-component of $\vec{y}_i$, for any $i \in I$.
\item[(3)] Return the pair $(C,X) \coloneqq (\vec{c} \frown C', \vec{x} \frown X')$, where $(C',X')$ is a solution for $(D_{-}-1,(\vec{y}_i - c_i \vec{x})_{i \in I})$ and $\frown$ denotes concatenation of a row vector to a matrix.
\ei
We identify $(C,X)$ with the sequence of pairs $(\vec{c},\vec{x}) \in O_k^I \times O_k^D$ in the way compatible with the construction above. We define {\it the score} of a pair $(\vec{c},\vec{x}) = ((c_i)_{i \in I},(x_d)_{d=0}^{D-1})$ as
\be
\left( \sum_{i \in I} \sum_{d=0}^{D-1} \v{c_i x_d}^q \right)^{1/q} = \left( \sum_{i \in I} \v{c_i}^q \right)^{1/q} \left( \sum_{d=0}^{D-1} \v{x_d}^q \right)^{1/q},
\ee
i.e.\ the $\ell^q$-matrix norm of ${}^{\r{t}}\vec{c} \vec{x} \coloneqq ((c_i x_d)_{d=0}^{D-1})_{i \in I} \in (O_k^D)^I$. We call the composite of the recursive method above (for some method to choose $\vec{x}$ in the step (2)) and sorting $(C,X)$ in the descending order of the score a {\it $p$-adic PCA}. Here is a pseudocode of the computation of $\vec{c}$ in the step (2) and the update of $Y$:

\begin{figure}[H]
\begin{algorithm}[H]
\caption{Recursion body of a $p$-adic PCA for $Y = (\vec{y}_i)_{i \in I} \in (\N_{p^E}^D)^I$}
\label{recursion body}
\begin{algorithmic}[1]
\Function {PCABody}{$p,D,E,Y,C,X,\vec{x}$}
	\State $\vec{c} = (c_i)_{i \in I} \gets (0)_{i \in I}$
	\ForAll {$i \in I$}
		\State $c_i \gets$ \Call{TrieTreeDFS}{$p,D,E,\vec{y}_i,\vec{x}$}
	\EndFor
	\If {$\vec{c}$ is a zero vector}
		\State \Return $(0,Y,C,X)$ \Comment{failure of update}
	\EndIf
	\ForAll {$i \in I$}
		\State $\vec{y}_i \gets (\vec{y}_i - c_i \vec{x}) \bmod p^E$
	\EndFor
	\State Append $\vec{c}$ to $C$
	\State Append $\vec{x}$ to $X$
	\State \Return $(1,Y,C,X)$ \Comment{success of update}
\EndFunction
\end{algorithmic}
\end{algorithm}
\end{figure}

The heuristic solution $(C,X)$ obtained by a $p$-adic PCA is not necessarily approximately optimal unless $\vec{x}$ is chosen appropriately in the step (2), because an optimal solution $(C,X)$ for $(D_{-},Y)$ before sorting is not necessarily a greedy solution, i.e.\ the pair $(\vec{c},\vec{x})$ of the first entries $\vec{c}$ and $\vec{x}$ of $C$ and $X$ respectively is not necessarily an optimal solution for $(1,Y)$. We will introduce in \S \ref{Non-reduced p-adic PCA} and \S \ref{Reduced p-adic PCA} two specific methods to choose $\vec{x}$ in the step (2), which we call {\it non-reduced $p$-adic PCA} and {\it reduced $p$-adic PCA}.

\subsection{Non-reduced $p$-adic PCA}
\label{Non-reduced p-adic PCA}

{\it Non-reduced $p$-adic PCA} is the $p$-adic PCA which chooses $\vec{x}$ to be the first non-zero entry of $Y$ (if exists). We note that $Y$ changes through the recursive process, and hence the whole $X$ is not computed before the recursive step. Here is a pseudocode for the process:

\begin{figure}[H]
\begin{algorithm}[H]
\caption{Non-reduced $p$-adic PCA for $Y = (\vec{y}_i)_{i \in I} \in (\N_{p^E}^D)^I$}
\label{non-reduced PCA}
\begin{algorithmic}[1]
\Function {NRPCA}{$p,D,E,Y,D_{-}$}
	\State $C \gets$ the empty array
	\State $X \gets$ the empty array
	\ForAll {$j \in I$}
		\If {$D_{-} = 0$}
			\State \Break
		\EndIf
		\If {$\vec{y}_j$ is a zero vector}
			\State \Continue
		\EndIf
		\State $\vec{x} \gets \vec{y}_j$
		\State $(b,Y,C,X) \gets $ \Call{PCABody}{$p,D,E,Y,C,X,\vec{x}$}
		\State $D_{-} \gets D_{-} - b$
	\EndFor
	\State Sort $(C,X)$ in the descending order of the score
	\State \Return $(C,X)$
\EndFunction
\end{algorithmic}
\end{algorithm}
\end{figure}

Here, $\vec{w} \bmod p^E$ denotes $(w_d \bmod p^E)_{d=0}^{D-1}$ for a $\vec{w} = (w_d)_{d=0}^{D-1} \in \Zp^D$. As an optional improvement of the scores of entries of $(C,X)$, it is good to execute \Call{NRPCA}{$p,D,E,Y,D'_{-}$} for a $D'_{-} \in \N_{\leq D}$ possibly bigger than $D_{-}$ and pop the last entries of $(C,X)$ while its length exceeds $D_{-}$.

\subsection{Reduced $p$-adic PCA}
\label{Reduced p-adic PCA}

By Example \ref{non-stability of orthogonality}, the return value $(C,X)$ of non-reduced $p$-adic PCA does not necessarily satisfy the property that the coordinate system $X$ is an orthogonal system. The redundancy based on the non-orthogonality of $X$ affects the largeness of the $\ell^q$-norm of rows $\vec{x}$ of $X$ relative to the $\ell^{\infty}$-norm, which prevents $(C,X)$ to be approximately optimal.

\vs
On the other hand, {\it reduced $p$-adic PCA} is the $p$-adic PCA which chooses $\vec{x}$ to be the first unused entry (if exists) of a pre-computed coordinate system $Z$ which is an approximately orthogonal system. In particular, candidates of entries of the whole $X$ are computed before the recursive step.

\vs
We explain the pre-computation of $Z$. First, we compute an iterated orthogonalisation $Y'$ of $Y \in V^I$. Next, we sort entries of $Y'$ in the descending order with respect to the norm. Then $Z$ is the resulting system, which belongs to $V^{\# I}$. Here is a pseudocode for the process:

\begin{figure}[H]
\begin{algorithm}[H]
\caption{Reduced $p$-adic PCA for $Y = (\vec{y}_i)_{i \in I} \in (\N_{p^E}^D)^I$}
\label{reduced PCA}
\begin{algorithmic}[1]
\Function {RPCA}{$p,D,E,Y,D_{-}$}
	\State $Y' \gets$ \Call{IteratedOrthogonalisation}{$p,D,E,Y$}
	\State $D'_{-} \gets$ the length of $Y'$
	\State $Z = (\vec{z}_d)_{d=0}^{D'_{-}-1} \gets$ the sequence given by sorting entries of $Y'$ in the descending order of the norm
	\State $C \gets$ the empty array
	\State $X \gets$ the empty array
	\ForAll {$d \in \N_{< D'_{-}}$}
		\If {$D_{-} = 0$ or $\vec{z}_d$ is a zero vector}
			\State \Break
		\EndIf
		\State $(b,Y,C,X) \gets $ \Call{PCABody}{$p,D,E,Y,C,X,\vec{z}_d$}
		\State $D_{-} \gets D_{-} - b$
	\EndFor
	\State Sort $(C,X)$ in the descending order of the score
	\State \Return $(C,X)$
\EndFunction
\end{algorithmic}
\end{algorithm}
\end{figure}

We list the main differences between non-reduced $p$-adic PCA and reduced $p$-adic PCA:
\bi
\item[(1)] The coordinate system $X$ is dynamically computed with $D_{-}$ recursive steps in a step-by-step way in non-reduced $p$-adic PCA, while candidates of entries of $X$ are pre-computed with typically around $\# I$ or $2 \# I$ steps in reduced $p$-adic PCA.
\item[(2)] The coordinate system $X$ is not an orthogonal system, while it is an approximately orthogonal system in reduced PCA.
\ei
In particular, there is a trade-off between the heaviness of the pre-computation of $X$ and the degree of approximate orthogonality of $X$.

\subsection{$p$-adic Line Search and $p$-adic Coordinate Descent}
\label{p-adic Line Search and p-adic Coordinate Descent}

By the definition of the orthogonality, $(C,X)$ is optimal only when every entry of the error term
\be
Y - {}^{\r{t}}CX \coloneqq \left( y_i - \sum_{d=0}^{D_{-}-1} c_{d,i} \vec{x}_d \right)_{i \in I} \in (O_k^D)^I
\ee
for $(C,X)$ is orthogonal to $\vec{c} X \coloneqq \sum_{d=0}^{D_{-}-1} c_d \vec{x}_d \in V$ for any $\vec{c} = (c_d)_{d=0}^{D_{-}-1} \in O_k^{D_{-}}$. Therefore, it is natural to ask whether a heuristic solution $(C,X)$ is {\it locally optimal} for a given finite subset $S$ of $O_k X \coloneqq \set{\vec{c} X}{\vec{c} \in O_k^{D_{-}}}$ in the sense that every entry of $Y - {}^{\r{t}}CX$ is orthogonal to any $\vec{x} \in S$.

\vs
Even if a given heuristic solution $(C,X)$ is not locally optimal for $S$, for any $\vec{x} \in S$, we may add to $C$ the coefficient matrix of the $\vec{x}$-components of entries of $Y - {}^{\r{t}}CX$ without increasing the loss $\n{Y - {}^{\r{t}}CX}$ by the definition of the notion of an $\vec{x}$-component.

\vs
A {\it $p$-adic line search} for $S$ is repetition of such updates of $C$ until $(C,X)$ becomes locally optimal for $S$ or the number of the repetition reaches a fixed threshold $T_{\LS}$. A {\it $p$-adic coordinate descent} is a $p$-adic line search for $\set{\vec{x}_d}{d \in \N_{< D_{-}}}$. Through a $p$-adic line search for $S$, we may expect that $(C,X)$ is approximately locally optimal for $S$. Here is a pseudocode for a $p$-adic line search:

\begin{figure}[H]
\begin{algorithm}[H]
\caption{A $p$-adic line search for $\set{(\theta_{d,h})_{d=0}^{D_{-}-1} X}{h \in \N_{< H}}$ for a $\Theta = ((\theta_{d,h})_{h=0}^{H-1})_{d=0}^{D_{-}-1} \in (O_k^H)^{D_{-}}$ with $H \in \N$ applied to $C = ((c_{d,i})_{i \in I})_{d=0}^{D_{-}-1}$}
\label{line search}
\begin{algorithmic}[1]
\Function {LineSearch}{$p,D,E,Y,C,X,\Theta$}
	\State $Y = (\vec{y}_i)_{i \in I} \gets Y - {}^{\r{t}}CX$
	\State $h \gets 0$
	\State $\temp \gets -1$
	\ForAll {$t \in \N_{< T_{\LS}}$}
		\State $\vec{x} \gets (\theta_{d,h})_{d=0}^{D_{-1}-1} X$
		\State $b \gets$ False \Comment{variable monitoring update of $C$}
		\ForAll {$i \in I$}
			\State $c \gets$ \Call{TrieTreeDFS}{$p,D,E,\vec{y}_i,\vec{x}$}
			\If {$c = 0$}
				\State \Continue
			\EndIf
			\State $b \gets$ True
			\ForAll {$d \in \N_{< D_{-}}$}
				\State $c_{d,i} \gets (c_{d,i} + c \theta_{d,h}) \bmod p^E$
			\EndFor
			\State $\vec{y}_i \gets (\vec{y}_i - c \vec{x}) \bmod p^E$
		\EndFor
		\If {$b$}
			\State $\temp \gets t$
			\If {$h = H - 1$}
				\State $h \gets 0$
			\Else
				\State $h \gets h + 1$
			\EndIf
		\ElsIf {$t = \temp + H$}
			\State \Break \Comment{the case where $(C,X)$ is locally optimal}
		\EndIf
	\EndFor
	\State \Return $C$
\EndFunction
\end{algorithmic}
\end{algorithm}
\end{figure}

A $p$-adic line search is useful to check whether a given heuristic solution is approximately optimal to some extent or not, and hence to check whether a given heuristic method works well or not.

\vs
In particular, non-reduced $p$-adic PCA and reduced $p$-adic PCA returned approximately locally optimal solutions in experiments for both of $S_0 = \set{\vec{x}_d}{d \in \N_{< D_{-}}}$ and $S_1 = \set{\vec{c} X}{\vec{c} \in S'}$ with a random subset $S' \subset O_k^{D_{-}}$ of cardinality $20$. More precisely, when we applied $p$-adic line searches for $S_0$ and $S_1$ in this order to heuristic solutions $(C,X)$ obtained by non-reduced $p$-adic PCA and reduced $p$-adic PCA, the former coordinate descent only non-significantly decreased the loss $\n{Y - {}^{\r{t}}CX}$, and the latter line search did not change the loss $\n{Y - {}^{\r{t}}CX}$, i.e.\ $(C,X)$ already became locally optimal for $S_1$. Therefore we refer only to the experimental results of non-reduced $p$-adic PCA and reduced $p$-adic PCA without $p$-adic line searches in \S \ref{Experiment}.

\section{Experiment}
\label{Experiment}

We abbreviate non-reduced $p$-adic PCA to {\it NRPCA}, and reduced $p$-adic PCA to {\it RPCA}. We compare the two $p$-adic PCAs for the case $p = 7$, $D = 100$, $E = 5$, $k = \Qp$ with normalisation $\v{p} = p^{-1}$, $V = \Qp^D$ equipped with the $\ell^1$-norm, $I = \N_{< 10000}$, and $D_{-} = 20$. For this purpose, we consider anomaly detection tasks where the $\ell^{\infty}$-norms of some of normal data can be smaller than those of many of anomalous data. The reason why we consider this $\ell^{\infty}$-norm condition is because dimensionality reduction by matrix factorisation based on Smith normal form, which recognises data of large $\ell^{\infty}$-norm as significant ones, does not work for such tasks.

\vs
Since we only use $p$-adic values modulo $p^E$, when we refer to a random point in $\Zp^d$ for a $d \in \N$, we consider a random point in the finite set $\N_{< p^E}^d$.

\vs
We denote by $S \subset \Zp^D$ the subset of normal points. We take an $r \in [0,100)$, and call it {\it the anomaly rate}. A {\it possibly anomalous point} is a point in $\Zp^D$ randomly taken from $S$ in probability $(100 - r) \%$ and from $\Zp^D$ in probability $r \%$. We generate the test data matrix $Y \in (\Zp^D)^I$ as a sequence of $\# I = 10000$ possibly anomalous points.

\vs
For a non-empty subset $I' \subset I$, we define {\it the compress ratio} $r_{\rC}(I')$ of $I'$ as
\be
1 - \frac{\sum_{i \in I'} \n{\vec{y}_i - \sum_{d=0}^{D_{-}-1} c_{d,i} \vec{x}_d}}{\sum_{i \in I'} \n{\vec{y}_i}} \in [0,1],
\ee
and {\it the deduced anomaly ratio} $r_{\rA}(I')$ of $I'$ as
\be
\frac{\# \set{i \in I'}{r_{\rC}(\ens{i}) < \epsilon_{\AD}}}{\# I'} \in [0,1]
\ee
using the return value $(C,X) = (((c_{d,i})_{i \in I})_{d=0}^{D_{-1}-1},(\vec{x}_d)_{d=0}^{D_{-}-1})$ of a $p$-adic PCA, where $\epsilon_{\AD} \in (0,1]$ is a fixed threshold. In particular, We deduce that an $i \in I$ is an index of anomalous data if the compress ratio
\be
r_{\rC}(\ens{i}) = 1 - \frac{\n{y_i - \sum_{d=0}^{D_{-}-1} c_{i,d} \vec{x}_d}}{\n{y_i}}
\ee
of the singleton $\ens{i}$ is smaller than $\epsilon_{\AD}$. In this section, we specifically set $\epsilon_{\AD} \coloneqq 0.2$.

\vs
We measure the performance of the two $p$-adic PCAs by the following amount:
\bi
\item[(1)] The false positive ratio, i.e.\ the deduced anomaly ratio for normal indices.
\item[(2)] The true positive ratio, i.e.\ the deduced anomaly ratio for anomalous indices.
\item[(3)] The compress ratio for normal indices.
\item[(4)] The compress ratio for anomalous indices.
\ei
The scores (1) and (4) are expected to be small, and the scores (2) and (3) are expected to be large.

\subsection{Open Balls}
\label{Open Balls}

We consider the case where $S$ is a disjoint union of closed balls of radius $\v{p}^2$. Let $B \in \N$ denote the number of closed balls. We label the balls by $\N_{< B}$. We call the balls corresponding to even labels (resp.\ odd labels) {\it even balls} (resp.\ {\it odd balls}). We choose the centres of even balls randomly from $\Zp^D$, and those of odd balls randomly from $p \Zp^D$.

\vs
We exhibit experiment results in a table of the following type:

\begin{table}[H]
\begin{center}
\caption{Algorithm name (values of $B$,$r$)}
\begin{tabular}{|c||c|c|c|c|}
\hline
\diagbox{I}{O} & A & N & $r_{\rA}$ & $r_{\rC}$ \\
\hline \hline
A & the number of & the number of & $r_{\rA}$ for & $r_{\rC}$ for \\
 & anomalous inputs & anomalous inputs & anomalous inputs & anomalous inputs \\
 & deduced to & deduced to & & \\
 & be normal & be anomalous & & \\
\hline
ball $0$ & the number of & the number of & $r_{\rA}$ for & $r_{\rC}$ for \\
 & normal inputs & normal inputs & normal inputs & normal inputs \\
 & in the ball $0$ & in the ball $0$ & in the ball $0$ & in the ball $0$ \\
 & deduced to & deduced to & & \\
 & be normal & be anomalous & & \\
\hline
$\vdots$ & $\vdots$ & $\vdots$ & $\vdots$ & $\vdots$ \\
\hline
ball $B-1$ & the number of & the number of & $r_{\rA}$ for & $r_{\rC}$ for \\
 & normal inputs & normal inputs & normal inputs & normal inputs \\
 & in the ball $B-1$ & in the ball $B-1$ & in the ball $B-1$ & in the ball $B-1$ \\
 & deduced to & deduced to & & \\
 & be normal & be anomalous & & \\
\hline
N & the number of & the number of & $r_{\rA}$ for & $r_{\rC}$ for \\
 & normal inputs & normal inputs & normal inputs & normal inputs \\
 & deduced to & deduced to & & \\
 & be normal & be anomalous & & \\
\hline
\end{tabular}
\end{center}
\end{table}

\vs
Here are results for $(B,r) = (10,1)$, $(10,10)$, $(30,1)$, and $(30,10)$:

\begin{table}[H]
\begin{center}
\caption{NRPCA ($B = 10 , r = 1)$}
\begin{tabular}{|c||c|c|c|c|}
\hline
\diagbox{I}{O} & A & N & $r_{\rA}$ & $r_{\rC}$ \\
\hline \hline
A & $86$ & $22$ & $0.80$ & $0.17$ \\
\hline
ball $0$ & $0$ & $1059$ & $0.00$ & $0.98$ \\
\hline
ball $1$ & $0$ & $944$ & $0.00$ & $0.89$ \\
\hline
ball $2$ & $0$ & $995$ & $0.00$ & $0.98$ \\
\hline
ball $3$ & $33$ & $940$ & $0.03$ & $0.66$ \\
\hline
ball $4$ & $0$ & $1013$ & $0.00$ & $0.98$ \\
\hline
ball $5$ & $0$ & $1021$ & $0.00$ & $0.87$ \\
\hline
ball $6$ & $0$ & $969$ & $0.00$ & $0.98$ \\
\hline
ball $7$ & $0$ & $977$ & $0.00$ & $0.88$ \\
\hline
ball $8$ & $0$ & $960$ & $0.00$ & $0.98$ \\
\hline
ball $9$ & $0$ & $981$ & $0.00$ & $0.80$ \\
\hline
N & $33$ & $9859$ & $0.00$ & $0.96$ \\
\hline
\end{tabular}

\caption{RPCA ($B = 10 , r = 1)$}
\begin{tabular}{|c||c|c|c|c|}
\hline
\diagbox{I}{O} & A & N & $r_{\rA}$ & $r_{\rC}$ \\
\hline \hline
A & $108$ & $0$ & $1.00$ & $0.10$ \\
\hline
ball $0$ & $0$ & $1059$ & $0.00$ & $0.51$ \\
\hline
ball $1$ & $298$ & $646$ & $0.32$ & $0.58$ \\
\hline
ball $2$ & $0$ & $995$ & $0.00$ & $0.98$ \\
\hline
ball $3$ & $0$ & $973$ & $0.00$ & $0.87$ \\
\hline
ball $4$ & $0$ & $1013$ & $0.00$ & $0.98$ \\
\hline
ball $5$ & $0$ & $1021$ & $0.00$ & $0.86$ \\
\hline
ball $6$ & $0$ & $969$ & $0.00$ & $0.98$ \\
\hline
ball $7$ & $977$ & $0$ & $1.00$ & $0.09$ \\
\hline
ball $8$ & $0$ & $960$ & $0.00$ & $0.91$ \\
\hline
ball $9$ & $0$ & $981$ & $0.00$ & $0.87$ \\
\hline
N & $1275$ & $8617$ & $0.13$ & $0.84$ \\
\hline
\end{tabular}
\end{center}
\end{table}

\begin{table}[H]
\begin{center}
\caption{NRPCA ($B = 10 , r = 10)$}
\begin{tabular}{|c||c|c|c|c|}
\hline
\diagbox{I}{O} & A & N & $r_{\rA}$ & $r_{\rC}$ \\
\hline \hline
A & $379$ & $695$ & $0.35$ & $0.21$ \\
\hline
ball $0$ & $0$ & $909$ & $0.00$ & $0.98$ \\
\hline
ball $1$ & $0$ & $868$ & $0.00$ & $0.64$ \\
\hline
ball $2$ & $0$ & $920$ & $0.00$ & $0.98$ \\
\hline
ball $3$ & $0$ & $894$ & $0.00$ & $0.60$ \\
\hline
ball $4$ & $0$ & $850$ & $0.00$ & $0.98$ \\
\hline
ball $5$ & $0$ & $907$ & $0.00$ & $0.62$ \\
\hline
ball $6$ & $0$ & $898$ & $0.00$ & $0.98$ \\
\hline
ball $7$ & $0$ & $895$ & $0.00$ & $0.32$ \\
\hline
ball $8$ & $0$ & $870$ & $0.00$ & $0.98$ \\
\hline
ball $9$ & $0$ & $915$ & $0.00$ & $0.72$ \\
\hline
N & $0$ & $8926$ & $0.00$ & $0.93$ \\
\hline
\end{tabular}

\caption{RPCA ($B = 10 , r = 10)$}
\begin{tabular}{|c||c|c|c|c|}
\hline
\diagbox{I}{O} & A & N & $r_{\rA}$ & $r_{\rC}$ \\
\hline \hline
A & $1068$ & $6$ & $0.99$ & $0.11$ \\
\hline
ball $0$ & $0$ & $909$ & $0.00$ & $0.98$ \\
\hline
ball $1$ & $0$ & $868$ & $0.00$ & $0.87$ \\
\hline
ball $2$ & $0$ & $920$ & $0.00$ & $0.37$ \\
\hline
ball $3$ & $135$ & $759$ & $0.15$ & $0.75$ \\
\hline
ball $4$ & $0$ & $850$ & $0.00$ & $0.98$ \\
\hline
ball $5$ & $0$ & $907$ & $0.00$ & $0.88$ \\
\hline
ball $6$ & $0$ & $898$ & $0.00$ & $0.98$ \\
\hline
ball $7$ & $0$ & $895$ & $0.00$ & $0.85$ \\
\hline
ball $8$ & $0$ & $870$ & $0.00$ & $0.98$ \\
\hline
ball $9$ & $0$ & $915$ & $0.00$ & $0.63$ \\
\hline
N & $135$ & $8791$ & $0.02$ & $0.85$ \\
\hline
\end{tabular}
\end{center}
\end{table}

\begin{table}[H]
\begin{center}
\caption{NRPCA ($B = 30 , r = 1)$}
\begin{tabular}{|c||c|c|c|c|}
\hline
\diagbox{I}{O} & A & N & $r_{\rA}$ & $r_{\rC}$ \\
\hline \hline
A & $38$ & $63$ & $0.38$ & $0.21$ \\
\hline
ball $0$ & $0$ & $341$ & $0.00$ & $0.59$ \\
\hline
ball $1$ & $164$ & $154$ & $0.52$ & $0.20$ \\
\hline
ball $2$ & $0$ & $316$ & $0.00$ & $0.54$ \\
\hline
ball $3$ & $2$ & $296$ & $0.01$ & $0.21$ \\
\hline
ball $4$ & $0$ & $353$ & $0.00$ & $0.52$ \\
\hline
ball $5$ & $0$ & $322$ & $0.00$ & $0.23$ \\
\hline
ball $6$ & $0$ & $344$ & $0.00$ & $0.71$ \\
\hline
ball $7$ & $0$ & $333$ & $0.00$ & $0.30$ \\
\hline
ball $8$ & $0$ & $320$ & $0.00$ & $0.47$ \\
\hline
ball $9$ & $0$ & $318$ & $0.00$ & $0.27$ \\
\hline
ball $10$ & $0$ & $328$ & $0.00$ & $0.58$ \\
\hline
ball $11$ & $0$ & $290$ & $0.00$ & $0.28$ \\
\hline
ball $12$ & $0$ & $351$ & $0.00$ & $0.62$ \\
\hline
ball $13$ & $0$ & $333$ & $0.00$ & $0.31$ \\
\hline
ball $14$ & $0$ & $335$ & $0.00$ & $0.60$ \\
\hline
ball $15$ & $106$ & $242$ & $0.30$ & $0.22$ \\
\hline
ball $16$ & $0$ & $322$ & $0.00$ & $0.62$ \\
\hline
ball $17$ & $0$ & $314$ & $0.00$ & $0.24$ \\
\hline
ball $18$ & $0$ & $369$ & $0.00$ & $0.58$ \\
\hline
ball $19$ & $0$ & $314$ & $0.00$ & $0.27$ \\
\hline
ball $20$ & $0$ & $321$ & $0.00$ & $0.57$ \\
\hline
ball $21$ & $0$ & $320$ & $0.00$ & $0.28$ \\
\hline
ball $22$ & $0$ & $325$ & $0.00$ & $0.60$ \\
\hline
ball $23$ & $16$ & $316$ & $0.05$ & $0.27$ \\
\hline
ball $24$ & $0$ & $321$ & $0.00$ & $0.51$ \\
\hline
ball $25$ & $375$ & $0$ & $1.00$ & $0.18$ \\
\hline
ball $26$ & $0$ & $330$ & $0.00$ & $0.54$ \\
\hline
ball $27$ & $0$ & $332$ & $0.00$ & $0.26$ \\
\hline
ball $28$ & $0$ & $317$ & $0.00$ & $0.53$ \\
\hline
ball $29$ & $0$ & $359$ & $0.00$ & $0.26$ \\
\hline
N & $663$ & $9236$ & $0.07$ & $0.53$ \\
\hline
\end{tabular}
\end{center}
\end{table}

\begin{table}[H]
\begin{center}
\caption{RPCA ($B = 30 , r = 1)$}
\begin{tabular}{|c||c|c|c|c|}
\hline
\diagbox{I}{O} & A & N & $r_{\rA}$ & $r_{\rC}$ \\
\hline \hline
A & $66$ & $35$ & $0.65$ & $0.18$ \\
\hline
ball $0$ & $0$ & $341$ & $0.00$ & $0.98$ \\
\hline
ball $1$ & $317$ & $1$ & $1.00$ & $0.15$ \\
\hline
ball $2$ & $0$ & $316$ & $0.00$ & $0.48$ \\
\hline
ball $3$ & $0$ & $298$ & $0.00$ & $0.26$ \\
\hline
ball $4$ & $0$ & $353$ & $0.00$ & $0.84$ \\
\hline
ball $5$ & $0$ & $322$ & $0.00$ & $0.29$ \\
\hline
ball $6$ & $0$ & $344$ & $0.00$ & $0.62$ \\
\hline
ball $7$ & $0$ & $333$ & $0.00$ & $0.33$ \\
\hline
ball $8$ & $0$ & $320$ & $0.00$ & $0.58$ \\
\hline
ball $9$ & $0$ & $318$ & $0.00$ & $0.30$ \\
\hline
ball $10$ & $0$ & $328$ & $0.00$ & $0.63$ \\
\hline
ball $11$ & $0$ & $290$ & $0.00$ & $0.33$ \\
\hline
ball $12$ & $0$ & $351$ & $0.00$ & $0.64$ \\
\hline
ball $13$ & $0$ & $333$ & $0.00$ & $0.30$ \\
\hline
ball $14$ & $0$ & $335$ & $0.00$ & $0.98$ \\
\hline
ball $15$ & $144$ & $204$ & $0.41$ & $0.21$ \\
\hline
ball $16$ & $0$ & $322$ & $0.00$ & $0.57$ \\
\hline
ball $17$ & $0$ & $314$ & $0.00$ & $0.29$ \\
\hline
ball $18$ & $0$ & $369$ & $0.00$ & $0.98$ \\
\hline
ball $19$ & $142$ & $172$ & $0.45$ & $0.20$ \\
\hline
ball $20$ & $0$ & $321$ & $0.00$ & $0.67$ \\
\hline
ball $21$ & $0$ & $320$ & $0.00$ & $0.24$ \\
\hline
ball $22$ & $0$ & $325$ & $0.00$ & $0.43$ \\
\hline
ball $23$ & $0$ & $332$ & $0.00$ & $0.25$ \\
\hline
ball $24$ & $0$ & $321$ & $0.00$ & $0.57$ \\
\hline
ball $25$ & $375$ & $0$ & $1.00$ & $0.18$ \\
\hline
ball $26$ & $0$ & $330$ & $0.00$ & $0.40$ \\
\hline
ball $27$ & $0$ & $332$ & $0.00$ & $0.30$ \\
\hline
ball $28$ & $0$ & $317$ & $0.00$ & $0.55$ \\
\hline
ball $29$ & $0$ & $359$ & $0.00$ & $0.34$ \\
\hline
N & $978$ & $8921$ & $0.10$ & $0.62$ \\
\hline
\end{tabular}
\end{center}
\end{table}

\begin{table}[H]
\begin{center}
\caption{NRPCA ($B = 30 , r = 10)$}
\begin{tabular}{|c||c|c|c|c|}
\hline
\diagbox{I}{O} & A & N & $r_{\rA}$ & $r_{\rC}$ \\
\hline \hline
A & $269$ & $719$ & $0.27$ & $0.23$ \\
\hline
ball $0$ & $0$ & $300$ & $0.00$ & $0.63$ \\
\hline
ball $1$ & $1$ & $285$ & $0.00$ & $0.22$ \\
\hline
ball $2$ & $0$ & $314$ & $0.00$ & $0.63$ \\
\hline
ball $3$ & $0$ & $265$ & $0.00$ & $0.24$ \\
\hline
ball $4$ & $0$ & $322$ & $0.00$ & $0.54$ \\
\hline
ball $5$ & $0$ & $338$ & $0.00$ & $0.24$ \\
\hline
ball $6$ & $0$ & $310$ & $0.00$ & $0.54$ \\
\hline
ball $7$ & $63$ & $233$ & $0.21$ & $0.23$ \\
\hline
ball $8$ & $0$ & $270$ & $0.00$ & $0.45$ \\
\hline
ball $9$ & $0$ & $280$ & $0.00$ & $0.26$ \\
\hline
ball $10$ & $0$ & $303$ & $0.00$ & $0.57$ \\
\hline
ball $11$ & $4$ & $333$ & $0.01$ & $0.21$ \\
\hline
ball $12$ & $0$ & $304$ & $0.00$ & $0.50$ \\
\hline
ball $13$ & $0$ & $276$ & $0.00$ & $0.28$ \\
\hline
ball $14$ & $0$ & $311$ & $0.00$ & $0.53$ \\
\hline
ball $15$ & $14$ & $281$ & $0.05$ & $0.21$ \\
\hline
ball $16$ & $0$ & $317$ & $0.00$ & $0.62$ \\
\hline
ball $17$ & $9$ & $258$ & $0.03$ & $0.21$ \\
\hline
ball $18$ & $0$ & $302$ & $0.00$ & $0.47$ \\
\hline
ball $19$ & $0$ & $279$ & $0.00$ & $0.24$ \\
\hline
ball $20$ & $0$ & $286$ & $0.00$ & $0.49$ \\
\hline
ball $21$ & $0$ & $300$ & $0.00$ & $0.27$ \\
\hline
ball $22$ & $0$ & $303$ & $0.00$ & $0.52$ \\
\hline
ball $23$ & $296$ & $0$ & $1.00$ & $0.14$ \\
\hline
ball $24$ & $0$ & $326$ & $0.00$ & $0.65$ \\
\hline
ball $25$ & $9$ & $333$ & $0.03$ & $0.24$ \\
\hline
ball $26$ & $0$ & $277$ & $0.00$ & $0.52$ \\
\hline
ball $27$ & $0$ & $295$ & $0.00$ & $0.24$ \\
\hline
ball $28$ & $0$ & $298$ & $0.00$ & $0.47$ \\
\hline
ball $29$ & $0$ & $317$ & $0.00$ & $0.27$ \\
\hline
N & $396$ & $8616$ & $0.04$ & $0.51$ \\
\hline
\end{tabular}
\end{center}
\end{table}

\begin{table}[H]
\begin{center}
\caption{RPCA ($B = 30 , r = 10)$}
\begin{tabular}{|c||c|c|c|c|}
\hline
\diagbox{I}{O} & A & N & $r_{\rA}$ & $r_{\rC}$ \\
\hline \hline
A & $964$ & $24$ & $0.98$ & $0.13$ \\
\hline
ball $0$ & $0$ & $300$ & $0.00$ & $0.98$ \\
\hline
ball $1$ & $286$ & $0$ & $1.00$ & $0.13$ \\
\hline
ball $2$ & $0$ & $314$ & $0.00$ & $0.98$ \\
\hline
ball $3$ & $265$ & $0$ & $1.00$ & $0.14$ \\
\hline
ball $4$ & $0$ & $322$ & $0.00$ & $0.98$ \\
\hline
ball $5$ & $338$ & $0$ & $1.00$ & $0.13$ \\
\hline
ball $6$ & $0$ & $310$ & $0.00$ & $0.98$ \\
\hline
ball $7$ & $296$ & $0$ & $1.00$ & $0.12$ \\
\hline
ball $8$ & $0$ & $270$ & $0.00$ & $0.25$ \\
\hline
ball $9$ & $280$ & $0$ & $1.00$ & $0.15$ \\
\hline
ball $10$ & $0$ & $303$ & $0.00$ & $0.98$ \\
\hline
ball $11$ & $337$ & $0$ & $1.00$ & $0.16$ \\
\hline
ball $12$ & $0$ & $304$ & $0.00$ & $0.24$ \\
\hline
ball $13$ & $276$ & $0$ & $1.00$ & $0.18$ \\
\hline
ball $14$ & $0$ & $311$ & $0.00$ & $0.98$ \\
\hline
ball $15$ & $295$ & $0$ & $1.00$ & $0.07$ \\
\hline
ball $16$ & $0$ & $317$ & $0.00$ & $0.98$ \\
\hline
ball $17$ & $267$ & $0$ & $1.00$ & $0.13$ \\
\hline
ball $18$ & $0$ & $302$ & $0.00$ & $0.98$ \\
\hline
ball $19$ & $279$ & $0$ & $1.00$ & $0.04$ \\
\hline
ball $20$ & $0$ & $286$ & $0.00$ & $0.27$ \\
\hline
ball $21$ & $300$ & $0$ & $1.00$ & $0.13$ \\
\hline
ball $22$ & $0$ & $303$ & $0.00$ & $0.98$ \\
\hline
ball $23$ & $296$ & $0$ & $1.00$ & $0.08$ \\
\hline
ball $24$ & $0$ & $326$ & $0.00$ & $0.98$ \\
\hline
ball $25$ & $342$ & $0$ & $1.00$ & $0.19$ \\
\hline
ball $26$ & $0$ & $277$ & $0.00$ & $0.94$ \\
\hline
ball $27$ & $295$ & $0$ & $1.00$ & $0.18$ \\
\hline
ball $28$ & $0$ & $298$ & $0.00$ & $0.98$ \\
\hline
ball $29$ & $317$ & $0$ & $1.00$ & $0.13$ \\
\hline
N & $4469$ & $4543$ & $0.50$ & $0.76$ \\
\hline
\end{tabular}
\end{center}
\end{table}

\vs
In all cases, RPCA showed larger positive ratios. This matches natural intuition that iterated orthogonalisation in the pre-computation significantly decreases the loss. Although NRPCA did not perform well, it showed lower false positive ratios. Therefore it can be a candidate when we need to focus on false positives rather than true positives.

\vs
It is notable that RPCA performs very well for the cases $B < D_{-}$, while neither NRPCA nor RPCA performs well for the cases $B > D_{-}$. Since normal points lie in the $\Zp$-submodule of $\Zp^D$ of rank $B$ up to $p^2 \Zp^D$, it is natural that the dimensionality reduction to dimension $D_{-}$ does not work well when $B > D_{-}$.

\vs
The results for RPCA for the cases $B > D_{-}$ show that the deduced anomaly ratios for odd balls are significantly higher than those for even balls, and that the compress ratios for odd balls are significantly smaller than those for even balls. This implies that RPCA recognises even balls as more significant components than odd balls.

\vs
On the other hand, as RPCA correctly detected anomalous data for the cases $B < D_{-}$, RPCA recognises odd balls as more significant components than anomalous data, despite of the largeness of the $\ell^{\infty}$-norm. Thus, RPCA actually succeeds in tasks which dimensionality reduction based on matrix factorisation by Smith normal form is not theoretically capable of.

\subsection{Affine Subspace}
\label{Affine Subspace}

We consider the case where $S$ is a subset of a positive codimensional affine subspace $W \subset V$ with possible noise of $\ell^{\infty}$-norm $\leq \v{p}^2$, i.e.\ $S \subset \set{\vec{y} \in V}{\exists \vec{w} \in W[\vec{y} - \vec{w} \in p^2 \Zp^D]}$. We denote by $D'$ the dimension of $W$.

\vs
In order to generate $W$, it suffices to choose generic $D'$ vectors from $\Zp^D$. We label them by $\N_{\leq D'}$, and denote by $B \in (\Zp^D)^{D'+1}$ the $(D \times (D'+1))$-matrix over $\Zp$ given by the sequence of the labeled vectors. We set
\be
X \coloneqq \set{\vec{y} \in \Zp^D}{\exists \vec{w} \in \Zp^{D'} \times \ens{1} \left[ \vec{y} - B \vec{w} \in p^2 \Zp^D \right]}.
\ee
We generate $B$ by randomly choosing even rows from $\Zp^D$ and randomly choosing odd rows from $p \Zp^D$.

\vs
We exhibit experiment results in a table of the following type:

\begin{table}[H]
\begin{center}
\caption{Algorithm name (values of $D'$,$r$)}
\begin{tabular}{|c||c|c|c|c|}
\hline
\diagbox{I}{O} & A & N & $r_{\rA}$ & $r_{\rC}$ \\
\hline \hline
A & the number of & the number of & $r_{\rA}$ for & $r_{\rC}$ for \\
 & anomalous inputs & anomalous inputs & anomalous inputs & anomalous inputs \\
 & deduced to & deduced to & & \\
 & be anomalous & be normal & & \\
\hline
N & the number of & the number of & $r_{\rA}$ for & $r_{\rC}$ for \\
 & normal inputs & normal inputs & normal inputs & normal inputs \\
 & deduced to & deduced to & & \\
 & be anomalous & be normal & & \\
\hline
\end{tabular}
\end{center}
\end{table}

\vs
Here are results for $(D',r) = (10,1)$, $(10,10)$, $(30,1)$, and $(30,10)$:

\begin{table}[H]
\begin{center}
\caption{NRPCA ($D' = 10 , r = 1)$}
\begin{tabular}{|c||c|c|c|c|}
\hline
\diagbox{I}{O} & A & N & $r_{\rA}$ & $r_{\rC}$ \\
\hline \hline
A & $7$ & $80$ & $0.08$ & $0.25$ \\
\hline
N & $1766$ & $8147$ & $0.18$ & $0.24$ \\
\hline
\end{tabular}

\caption{RPCA ($D' = 10 , r = 1)$}
\begin{tabular}{|c||c|c|c|c|}
\hline
\diagbox{I}{O} & A & N & $r_{\rA}$ & $r_{\rC}$ \\
\hline \hline
A & $85$ & $2$ & $0.98$ & $0.14$ \\
\hline
N & $653$ & $9260$ & $0.07$ & $0.27$ \\
\hline
\end{tabular}
\end{center}
\end{table}

\begin{table}[H]
\begin{center}
\caption{NRPCA ($D' = 10 , r = 10)$}
\begin{tabular}{|c||c|c|c|c|}
\hline
\diagbox{I}{O} & A & N & $r_{\rA}$ & $r_{\rC}$ \\
\hline \hline
A & $79$ & $852$ & $0.08$ & $0.25$ \\
\hline
N & $1392$ & $7677$ & $0.15$ & $0.25$ \\
\hline
\end{tabular}

\caption{RPCA ($D' = 10 , r = 10)$}
\begin{tabular}{|c||c|c|c|c|}
\hline
\diagbox{I}{O} & A & N & $r_{\rA}$ & $r_{\rC}$ \\
\hline \hline
A & $921$ & $10$ & $0.99$ & $0.13$ \\
\hline
N & $1097$ & $7972$ & $0.12$ & $0.26$ \\
\hline
\end{tabular}
\end{center}
\end{table}

\begin{table}[H]
\begin{center}
\caption{NRPCA ($D' = 30 , r = 1)$}
\begin{tabular}{|c||c|c|c|c|}
\hline
\diagbox{I}{O} & A & N & $r_{\rA}$ & $r_{\rC}$ \\
\hline \hline
A & $7$ & $86$ & $0.08$ & $0.25$ \\
\hline
N & $1735$ & $8172$ & $0.18$ & $0.25$ \\
\hline
\end{tabular}

\caption{RPCA ($D' = 30 , r = 1)$}
\begin{tabular}{|c||c|c|c|c|}
\hline
\diagbox{I}{O} & A & N & $r_{\rA}$ & $r_{\rC}$ \\
\hline \hline
A & $92$ & $1$ & $0.99$ & $0.14$ \\
\hline
N & $1548$ & $8359$ & $0.16$ & $0.25$ \\
\hline
\end{tabular}
\end{center}
\end{table}

\begin{table}[H]
\begin{center}
\caption{NRPCA ($D' = 30 , r = 10)$}
\begin{tabular}{|c||c|c|c|c|}
\hline
\diagbox{I}{O} & A & N & $r_{\rA}$ & $r_{\rC}$ \\
\hline \hline
A & $84$ & $895$ & $0.09$ & $0.25$ \\
\hline
N & $1378$ & $7643$ & $0.15$ & $0.25$ \\
\hline
\end{tabular}

\caption{RPCA ($D' = 30 , r = 10)$}
\begin{tabular}{|c||c|c|c|c|}
\hline
\diagbox{I}{O} & A & N & $r_{\rA}$ & $r_{\rC}$ \\
\hline \hline
A & $940$ & $39$ & $0.96$ & $0.15$ \\
\hline
N & $1426$ & $7595$ & $0.16$ & $0.25$ \\
\hline
\end{tabular}
\end{center}
\end{table}

\vs
In all cases, RPCA showed very high true positive ratios, while there are not significant differences in false positive ratios.

\vs
It is notable that RPCA performs very well not only for the cases $D' < D_{-}$, but also for the cases $D' > D_{-}$. Since normal points lie in the $\Zp$-submodule of $\Zp^D$ of rank $D'$ up to possible noise in $p^2 \Zp^D$, it is difficult to achieve such performance by purely linear algebraic methods. For example, dimensionality reduction by on matrix factorisation based on Smith normal form or Gauss elimination over $\Fp$ does not work in these unsupervised settings, as the $\ell^{\infty}$-norms of anomalous points are greater than those of odd rows of $B$.

% \newpage
\vspace{0.3in}
\addcontentsline{toc}{section}{Acknowledgements}
\noindent {\Large \bf Acknowledgements}
\vspace{0.2in}

\noindent
I thank all people who helped me to learn mathematics and programming. I also thank my family.
%
%
%\vspace{0.3in}
%\addcontentsline{toc}{section}{Compliance with Ethical Standards}
%\noindent {\Large \bf Compliance with Ethical Standards}
%\vspace{0.2in}
%
%\begin{description}
%\item[Funding:] We declare that there exists no funding supporting us.
%\item[Conflict of Interest:] We declare that there exists no conflict of interest.
%\item[Author Contribution:] We declare that we substantially contributed to the entire study, including the conception and design of the study, acquisition of data if applicable, analysis and interpretation of data if applicable, drafting and revising the article critically for important intellectual content, and final approval of the version to be submitted.
%\item[Data Availability Statements:] We declare that we analysed no existing data set because our research proceeds a theoretic and mathematical approach. We simply input random numbers generated by ``randint'' method of ``random'' module in python standard library during the execution of an experiment code.
%\item[Duplicated Submission:] We declare that the manuscript, including related data if applicable, figures if applicable, and tables if applicable, is not and will not be under active consideration elsewhere for the duration of the review process, and has never been published in a journal or presented in a poster session. 
%\end{description}

\end{document}